\documentclass[11pt]{amsart}
\usepackage{amssymb,mathrsfs,graphicx,enumerate,color}
\usepackage{colortbl}
\definecolor{black}{rgb}{0.0, 0.0, 0.0}
\definecolor{red}{rgb}{1.0, 0.5, 0.5}

\title[   ]{Global smooth solutions for 1D barotropic Navier-Stokes equations with a large class of degenerate viscosities}

\author[Kang]{Moon-Jin Kang}
\address[Moon-Jin Kang]{\newline Department of Mathematic \& Research Institute of Natural Sciences, \newline Sookmyung Women's University, Seoul 140-742, Korea}
\email{moonjinkang@sookmyung.ac.kr}

\author[Vasseur]{Alexis F. Vasseur}
\address[Alexis F. Vasseur]{\newline Department of Mathematics, \newline The University of Texas at Austin, Austin, TX 78712, USA}
\email{vasseur@math.utexas.edu}

\newtheorem{theorem}{Theorem}[section]
\newtheorem{lemma}{Lemma}[section]

\newtheorem{proposition}{Proposition}[section]
\newtheorem{remark}{Remark}[section]

\newcommand{\bbr}{\mathbb R}

\numberwithin{figure}{section}
\newcommand{\beq}{\begin{equation}}
\newcommand{\eeq}{\end{equation}}
\newcommand{\bsp}{\begin{split}}
\newcommand{\esp}{\end{split}}

\def\eps{\varepsilon }

\newcommand\adots{\mathinner{\mkern2mu\raise1pt\hbox{.}
\mkern3mu\raise4pt\hbox{.}\mkern1mu\raise7pt\hbox{.}}}

\def\charf {\mbox{{\text 1}\kern-.30em {\text l}}}

\newcommand{\rhoe}{\rho_\eps}
\newcommand{\ue}{u_\eps}
\newcommand{\we}{w_\eps}
\newcommand{\mue}{\mu_\eps}
\newcommand{\px}{\partial_x}

\begin{document}
\bibliographystyle{plain}

\date{\today}

\subjclass[2010]{35Q35, 76N10} 
\keywords{Existence, Uniqueness, Smooth solution, 1D barotropic Navier-Stokes system, Degenerate viscosity}

\thanks{\textbf{Acknowledgment.}  The first author was partially supported by the NRF-2019R1C1C1009355.
The second author was partially supported by the NSF grant: DMS 1614918. 
}

\begin{abstract}
We prove the global existence and uniqueness of smooth solutions to the one-dimensional barotropic Navier-Stokes system with degenerate viscosity $\mu(\rho)=\rho^\alpha$.
We establish that the smooth solutions have possibly two different far-fields, and the initial density remains positive globally in time, for the initial data satisfying the same conditions. 
In addition, our result works for any $\alpha>0$, i.e., for a large class of degenerate viscosities.
In particular, our models include the viscous shallow water equations.
This extends the result of Constantin-Drivas-Nguyen-Pasqualotto \cite[Theorem 1.6]{CDNP} (on the case of periodic domain) to the case where smooth solutions connect possibly two different limits at the infinity on the whole space. 
  \end{abstract}
\maketitle \centerline{\date}

\tableofcontents

\section{Introduction}
\setcounter{equation}{0}

We consider the one-dimensional barotropic Navier-Stokes system in the Eulerian coordinates:
\begin{align}
\begin{aligned}\label{NS}
\left\{ \begin{array}{ll}
        \rho_t + (\rho u)_x =0,\\
       (\rho u)_t+(\rho u^2)_x + p(\rho)_x = (\mu(\rho) u_x)_x, \end{array} \right.
\end{aligned}
\end{align}
where the pressure $p(\rho)$ follows the case of a polytropic perfect gas, i.e.,
\beq\label{pressure}
p(\rho)= \rho^{\gamma},\quad \gamma> 1,
\eeq
with  $\gamma$ the adiabatic constant. Here, $\mu$ denotes the viscosity coefficient given by 
\beq\label{mu-def}
\mu(\rho) = \rho^{\alpha}. 
\eeq
Notice that if $\alpha>0$, $\mu(\rho)$ degenerates near the vacuum, i.e., near $\rho=0$.
Very often, the viscosity coefficient is assumed to be constant, i.e., $\alpha=0$. However, in the  physical context  the viscosity of a gas depends on the temperature (see Chapman and Cowling \cite{CC}).  In the barotropic case,  the viscosity depends directly on the density. 
In general, the viscosity  is expected to degenerate on the vacuum as a power of the density as in \eqref{mu-def}.\\

There are many results on the existence of solutions to the compressible Navier-Stokes equations with the constant viscosity for the one-dimensional case.
The existence of weak solutions was first established by Kazhikhov and Shelukhin \cite{KS} for smooth enough initial data close to the equilibrium bounded away from zero. The case of discontinuous data but still bounded away from zero was addressed by Shelukhin \cite{Shel82,Shel83,Shel84} and then by Serre \cite{Serre86} and Hoff \cite{Hoff87}. First result for vanishing initial density was obtained by Shelukhin \cite{Shel86}. Hoff \cite{Hoff98} proved the existence of global weak solutions with large discontinuous initial data, possibly having different limits at the infinity. There, he also proved that the vacuum cannot form in finite time.
The issues on regularity and uniqueness of solutions was first studied by Solonnikov \cite{Solo} for smooth initial data and for small time. However, the regularity may blow-up as the solution gets close to vacuum. Hoff and Smoller \cite{HoSm} show that any weak solution of the one-dimensional Navier-Stokes equations do not have vacuum states for every time, provided that no vacuum states initially exist.
 
Concerning the 1D existence theory for the degenerate case \eqref{NS}, Mellet-Vasseur \cite{MV_sima} proved the global existence and uniqueness of strong solutions  
with large initial data having possibly different limits at the infinity without no vacuum states in the case of $\alpha<1/2$ and $\gamma>1$. 
To control the $L^\infty$-norm of $1/\rho$ globally in time, they used the relative entropy inequality based on the Bresch-Desjardins entropy, which was derived in \cite{BD_Paris02} for the multi-dimensional Korteweg system of equations (for the case of $\alpha=1$ and with an additional capillary term) and later generalized in \cite{BD_Paris04}. In the one-dimensional case, a similar inequality was introduced earlier by Vaigant \cite{Vai} for flows with constant viscosity.

The result of Mellet-Vasseur \cite{MV_sima} was extended by Haspot \cite{Haspot} to the case of $\alpha\in (1/2,1]$. Recently, Constantin-Drivas-Nguyen-Pasqualotto \cite[Theorem 1.6]{CDNP} extended it to the case of $\alpha\ge 0$ and $\gamma\in [\alpha, \alpha+1]$ with $\gamma>1$, but they dealt with it on the periodic domain, and with an additional technical condition (see \eqref{mono-w0}). \\

In this article, we aim to extend the result \cite[Theorem 1.6]{CDNP} to the case where smooth solutions have possibly different limits at the infinity on the whole space.  
This extended result is motivated by the recent works \cite{Kang-V-NS17,KV-unique19} of the authors on the contraction property for any large perturbations of viscous shocks of the one-dimensional barotropic Navier-Stokes system with degenerate viscosity. 

\subsection{Main results}
We study global existence of smooth solutions to \eqref{NS} with initial data having possibly two different limits $(\rho_\pm, u_\pm)$ at $x=\pm \infty$, where $\rho_\pm>0$.
For that, we let $\bar \rho$ and $\bar u$ be smooth monotone functions such that
\beq\label{smooth-end}
\bar\rho (x)=\rho_\pm >0 \quad\mbox{and}\quad \bar u(x)=u_\pm,\quad \mbox{when } \pm x\ge 1.
\eeq

\begin{theorem}\label{thm:main}
Assume $\gamma>1, \alpha>0, $ and $\gamma\in [\alpha,\alpha+1]$. Let $\rho_0$ and $u_0$ be the initial data  such that 
\begin{align}
\begin{aligned} \label{ass-ini}
& \rho_0 -\bar \rho \in H^k(\bbr),\qquad u_0 -\bar u \in H^k(\bbr),\qquad \mbox{for some integer } k\ge 4,\\
&0<\underline \kappa_0 \le \rho_0(x) \le \overline \kappa_0,\quad \forall x\in\bbr,\qquad \mbox{for some constants $\underline \kappa_0, \overline \kappa_0$},
\end{aligned}
\end{align}
and
\beq\label{mono-w0}
\partial_x u_0(x)\le \rho_0(x)^{\gamma-\alpha},\qquad \forall x\in\bbr,
\eeq 
where $\bar\rho$ and $\bar u$ are the smooth monotone functions satisfying \eqref{smooth-end}.\\ 
Then there exists a global-in-time unique smooth solution $(\rho, u)$ of \eqref{NS}-\eqref{mu-def} such that 
for any $T>0$,
\begin{align*}
\begin{aligned}
&\rho_\eps -\bar \rho \in L^\infty(0,T;H^k(\bbr)) \\
& u_\eps-\bar u \in L^\infty(0,T;H^k(\bbr)) \cap  L^2(0,T;H^{k+1}(\bbr)).
\end{aligned}
\end{align*}
Moreover, there exists constants $\underline \kappa(T)$ and $\overline \kappa(T)$ such that
\[
\underline  \kappa(T) \le \rho(t,x) \le \overline  \kappa(T),\qquad \forall (t,x)\in [0,T]\times \bbr.
\]
\end{theorem}

\begin{remark}
Note that the system \eqref{NS} is equivalent to the one in the mass Lagrangian coordinates for the regularity in Theorem \ref{thm:main}. Therefore, the above result provides a class of global-in-time solutions smooth enough, in which the authors proved the contraction property \cite{Kang-V-NS17,KV-unique19} for viscous shocks of the barotropic Navier-Stokes system in the mass Lagrangian coordinates, with any large initial data satisfying \eqref{ass-ini} and \eqref{mono-w0}.
\end{remark}

\begin{remark}
Note from the assumption on $\alpha$ and $\gamma$ that Theorem \ref{thm:main} also holds for the viscous shallow water equations (i.e., $\gamma=2$, $\alpha=1$). We refer to Gerbeau-Perthame \cite{GP} for a derivation of the viscous shallow water equations from the incompressible Navier-Stokes equations with free boundary.
\end{remark}

\begin{remark}
The initial assumptions on \eqref{mono-w0} and $k\ge 4$ in \eqref{ass-ini} are the same conditions as in  \cite[Theorem 1.5]{CDNP}, which is used to control the active potential \eqref{def-w} defined by the density and the velocity (see Lemma \ref{lem:up-w}).
\end{remark}

\section{Proof of Theorem \ref{thm:main}}
\setcounter{equation}{0}

\subsection{Idea of Proof}
Since we are looking for solutions converging to possibly two different limits $(\rho_\pm, u_\pm)$ at $x=\pm \infty$, we do not expect that solutions are integrable. 
Thus, as a starting point, we may take advantage of the existence result \cite{MV_sima}, for solutions $(\rho, u)$ to satisfy $\rho-\bar\rho, u-\bar u\in L^\infty(0,T;L^2(\bbr))$.
However, since the result \cite{MV_sima} require the assumption $\alpha<1/2$ while we consider any $\alpha>0$, we may perturb the viscosity coefficient \eqref{mu-def} by adding $\eps\rho^{1/4}$ with small parameter $\eps$ as in \eqref{new-v}, under which we ensure the global existence of strong solution $(\rhoe,\ue)$ satisfying the $H^1$-spatial regularity and the positive lower-bound of the density (see \eqref{eps-reg} and \eqref{eps-bdd}).\\
To remove the $\eps$-dependence of the approximate viscosity $\mue$ as in \eqref{ind-mu}, we may first show that the lower bound of the density $\rhoe$ is independent of $\eps$ as in Proposition \ref{lem:rho2}. For that, we basically use the idea in \cite{CDNP} on the analysis for the time-evolution of the active potential (see Lemma \ref{lem:up-w}). To perform the analysis, we need at least $H^4$-spatial regularity of $(\rhoe,\ue)$, which requires the initial condition \eqref{ass-ini}.\\

\subsection{Approximate viscosity}
As mentioned above, we first recall the existence result in \cite{MV_sima} as follows:

\begin{proposition}\cite{MV_sima} \label{prop:mv}
Let $\rho_0$ and $u_0$ be the initial data such that
\beq\label{prop:ini}
0<\underline \kappa_0 \le \rho_0(x) \le \overline \kappa_0, \quad \rho_0 -\bar \rho \in H^1(\bbr),\quad u_0 -\bar u \in H^1(\bbr),
\eeq
for some constants $\underline \kappa_0, \overline \kappa_0$. 
Let $\nu:\bbr_+\to\bbr_+$ be a function such that for some constants $C>0$ and $q\in [0,1/2)$,
\begin{align}
\begin{aligned} \label{oldv-1}
\nu(y)\ge \left\{ \begin{array}{ll}
       C y^q \qquad &\forall y\le 1\\
       C \qquad &\forall y\ge 1, \end{array} \right.
\end{aligned}
\end{align}
and 
\beq\label{oldv-2}
\nu(y)\le C + C y^\gamma \qquad \forall y\ge 0.
\eeq
Then there exists a global-in-time unique strong solution $(\rho, u)$ of \eqref{NS}-\eqref{pressure} with $\mu=\nu$ such that the following holds:\\
For any $T>0$, there exist positive constants $\underline \beta(T)$ and $\overline \beta(T)$ such that 
\begin{align*}
\begin{aligned} 
&\rho -\bar \rho \in L^\infty(0,T;H^1(\bbr)),\\
&u-\bar u\in L^\infty(0,T;H^1(\bbr))\cap L^2(0,T;H^2(\bbr)),\\
&\underline \beta(T) \le \rho(t,x) \le \overline \beta(T),\qquad \forall (t,x)\in [0,T]\times \bbr.
\end{aligned}
\end{align*}
\end{proposition}

\vspace{1cm}

To use Proposition \ref{prop:mv}, we consider an approximate viscosity coefficient $\mu_\eps$ defined by perturbing the viscosity $\mu$ in \eqref{mu-def} as follows: For any $0<\eps<1$,
\beq\label{new-v}
\mu_\eps (\rho) :=\max\left(\mu(\rho), \eps \rho^{\alpha_*}\right),\quad \forall \rho\ge0,\quad\mbox{where~} \alpha_*:=\frac{1}{2}\min\left(\alpha,\frac{1}{2} \right).
\eeq
Since  
\[
\mu_\eps (\rho) \ge \left\{ \begin{array}{ll}
       \eps \rho^{1/4} \qquad &\forall \rho\le 1\\
       \eps \qquad &\forall \rho\ge 1, \end{array} \right.
\]
and it follows from $\gamma\ge\alpha$ that 
\beq\label{up-mue}
\mu_\eps (\rho) \le 1+ \rho^\gamma\qquad \forall \rho\ge0,
\eeq
$\mu_\eps$ satisfies the assumptions \eqref{oldv-1} and \eqref{oldv-2}. Therefore, for the initial datum $(\rho_0, u_0)$ satisfying \eqref{ass-ini}, Proposition \ref{prop:mv} implies that there exists a global-in-time unique strong solution $(\rho_\eps, u_\eps)$ of \eqref{NS}-\eqref{pressure} with $\mu=\mu_\eps$, i.e., 
\begin{align}
\begin{aligned}\label{NS-eps}
\left\{ \begin{array}{ll}
        \partial_t \rhoe + \partial_x(\rhoe \ue) =0\\
        \partial_t (\rhoe \ue)+ \partial_x (\rhoe \ue^2) +  \partial_x p(\rhoe) =  \partial_x (\mue(\rhoe)  \partial_x \ue)\\
        (\rhoe, \ue) |_{t=0} =(\rho_0, u_0), \end{array} \right.
\end{aligned}
\end{align}
such that the following holds: For any $T>0$, there exist positive constants $\underline \kappa_\eps(T)$,  $\overline \kappa_\eps(T)$ and $C=C(T,\eps, \underline\kappa_0,\overline \kappa_0)$ such that 
\begin{align}
\begin{aligned}\label{eps-reg} 
\|\rho_\eps -\bar \rho \|_{L^\infty(0,T;H^1(\bbr))} + \| u_\eps-\bar u \|_{L^\infty(0,T;H^1(\bbr))} +  \| u_\eps-\bar u \|_{L^2(0,T;H^2(\bbr))} \le C,
\end{aligned}
\end{align}
and
\beq\label{eps-bdd}
\underline \kappa_\eps(T) \le \rho_\eps(t,x) \le \overline \kappa_\eps(T),\qquad \forall (t,x)\in (0,T)\times \bbr.
\eeq

\subsection{Higher Sobolev regularity}
For the system \eqref{NS-eps}, we consider the active potential 
\beq\label{def-w}
w_\eps := -p(\rho_\eps) + \mu_\eps(\rho_\eps) \partial_x u_\eps.
\eeq
This is the potential in the momentum equation of \eqref{NS-eps}. Indeed, its gradient is the force:
\[
\rhoe ( \partial_t\ue +\ue\px\ue)  = \px \we.
\]  
Then it follows from \cite[Proposition 3.1]{CDNP} that $w_\eps$ satisfies a forced quadratic heat equation with linear drift:
\begin{align}
\begin{aligned} \label{w-eq}
\partial_t \we &= \frac {\mue(\rhoe)}{\rhoe} \px^2 \we -\left(\ue+\mue(\rhoe)\frac{\px\rhoe}{\rhoe^2}\right) \px\we +\left( \rhoe\frac{p'(\rhoe)}{\mue(\rhoe)} -2p(\rhoe) \frac{\rhoe\mue'(\rhoe)+\mue(\rhoe)}{\mue(\rhoe)^2}  \right)\we \\
&\quad - \frac{\rhoe\mue'(\rhoe)+\mue(\rhoe)}{\mue(\rhoe)^2} \we^2 + \left( \rhoe\frac{p'(\rhoe)}{\mue(\rhoe)} -p(\rhoe)\frac{\rhoe\mue'(\rhoe)+\mue(\rhoe)}{\mue(\rhoe)^2}  \right) p(\rhoe).
\end{aligned}
\end{align}
Note that the new viscosity coefficient $\mue(\rhoe)/\rhoe$ of the parabolic equation \eqref{w-eq} on $\we$ is less degenerate than the viscosity coefficient $\mue(\rhoe)$ of the momentum equation in \eqref{NS-eps}. Through the coupled system of \eqref{w-eq} and the continuity equation $\eqref{NS-eps}_1$, we obtain the higher Sobolev regularity of $\rhoe$ and $\we$ as long as $\rhoe$ is positive (that is guaranteed by \eqref{eps-bdd}) as follows:

\begin{lemma} \label{lem:higher}
Let $\gamma,\alpha$ be any real numbers.
Assume that the initial data $\rho_0$ and $u_0$ satisfy
\begin{align}
\begin{aligned} \label{temp-ini}
&\rho_0 -\bar \rho \in H^k(\bbr),\quad u_0 -\bar u \in H^k(\bbr),\quad \mbox{for some integer } k\ge 2,\\
&0<\underline \kappa_0 \le \rho_0(x) \le \overline \kappa_0, \quad \forall x\in\bbr,
\end{aligned}
\end{align}
for some constants $\underline \kappa_0, \overline \kappa_0$. 
Then, there exists a global-in-time unique smooth solution $(\rho_\eps, u_\eps)$ of \eqref{NS-eps} such that the following holds: 
For any $T>0$, there exists positive constants $\underline \kappa_\eps(T)$, $\overline \kappa_\eps(T)$ and $C=C(T,\gamma,\alpha, k,\eps, \underline\kappa_0,\overline \kappa_0)$  such that  
\eqref{eps-reg}, \eqref{eps-bdd} and
\begin{align*}
\begin{aligned} 
\|\partial_x^k \rho_\eps \|_{L^\infty(0,T;L^2(\bbr))} &+ \|\partial_x^{k-1} w_\eps \|_{L^\infty(0,T;L^2(\bbr))}+\|\partial_x^{k} w_\eps \|_{L^2(0,T;L^2(\bbr))} \\
&\qquad + \|\partial_x^k u_\eps \|_{L^\infty(0,T;L^2(\bbr))} + \|\partial_x^{k+1} u_\eps \|_{L^2(0,T;L^2(\bbr))} \le C.
\end{aligned}
\end{align*}
\end{lemma}

\vspace{1cm}

This follows straightforwardly from \cite[Lemma 4.2 and 4.3]{CDNP} when $\|w_\eps\|_{L^\infty(0,T;L^2(\bbr))}$ is bounded. However, for the density having two different limits at the infinity, we do not have a $L^2$-bound on $w_\eps (t,x)$ for each $t$.
Therefore, we may prove Lemma \ref{lem:higher} without using a $L^2$-bound on $w_\eps$. Although we need a slight modification of the proof in \cite{CDNP}, we present details of the proof in Appendix \ref{app-higher} for the sake of completeness and the justification on uniformity of the high Sobolev norms in Proposition \ref{lem:all}.

\subsection{Uniform lower bound for the density}
\begin{lemma}\label{lem:up-w}
Assume the same hypotheses as in Theorem \ref{thm:main}.
Then, for any $T>0$, there exist positive constants $C_\gamma$ and $\eps_\gamma$ such that
\begin{align*}
\begin{aligned} 
w_\eps (t,x) \le C_\gamma \eps^\theta, \qquad \forall \eps\le \eps_\gamma, 
\quad \forall t\le T, \quad \forall x\in\bbr,
\end{aligned}
\end{align*}
where $\theta$ is the positive constant as follows:
\begin{align}
\begin{aligned} \label{special-c}
 \theta:=\frac{\gamma}{\alpha-\alpha_*},\qquad\mbox{where $\alpha_*$ is the constant as in \eqref{new-v}}.
\end{aligned}
\end{align}

\end{lemma}
\begin{proof}
First of all, using Lemma \ref{lem:higher} with $k\ge 4$, together with \eqref{NS-eps} and \eqref{def-w}, we have
\[
\rhoe, \ue, \we \in C^1([0,T]\times\bbr).
\]
Then, note from \eqref{def-w}, \eqref{new-v}, \eqref{pressure}, \eqref{mu-def} and the initial condition \eqref{mono-w0} that 
\[
\we(0,x) =-p(\rho_0) +\max\left(\mu(\rho_0), \eps \rho_0^{\alpha_*}\right) \partial_x u_0 \le -\rho_0^\gamma + \max\left(\rho_0^\alpha, \eps \rho_0^{\alpha_*}\right) \rho_0^{\gamma-\alpha}.
\]
Since, for all $x\in\bbr$,
\begin{align*}
\begin{aligned} 
\we(0,x) &\le \left(-\rho_0^\gamma + \rho_0^\alpha \rho_0^{\gamma-\alpha}\right) {\mathbf 1}_{\{\rho_0^\alpha>\eps\rho_0^{\alpha_*}\}} + \left(-\rho_0^\gamma +  \eps \rho_0^{\alpha_*} \rho_0^{\gamma-\alpha}\right) {\mathbf 1}_{\{\rho_0^\alpha\le \eps\rho_0^{\alpha_*}\}}\\
&\le    \eps \rho_0^{\gamma-(\alpha-\alpha_*)} {\mathbf 1}_{\{\rho_0^\alpha\le \eps\rho_0^{\alpha_*}\}} \le \eps^{\frac{\gamma}{\alpha-\alpha_*}},
\end{aligned}
\end{align*}
we have
\[
\we(0,x) \le \eps^\theta,\quad \forall x\in\bbr.
\]
Since $\we \in C([0,T]\times\bbr)$, if there exists a point $(t_0,x_0)\in (0,T]\times\bbr$ such that $\we(t_0,x_0)>\eps^\theta$, 
then there exists $t_1\ge 0$ such that 
\beq\label{neg-le}
\sup_{x\in\bbr} \we (t,x)\le \eps^\theta\quad \forall t\in [0,t_1],\\
\eeq
and
\[
\sup_{x\in\bbr} \we (t,x)> \eps^\theta\quad \forall t\in (t_1,t_0].
\]
Let 
\[
t_2:=\sup \left\{ t\in (t_1, T]~|~  \sup_{x\in\bbr} \we (t,x)> \eps^\theta \right\}.
\]
Then, 
\[
 \sup_{x\in\bbr} \we (t,x)\ge \eps^\theta\quad \forall t\in [t_1,t_2].
\]
Thus, using the fact that for each $t\le T$,
\[
\we (t,x) \to -p(\rho_\pm) \le 0 \quad\mbox{as } ~ x\to\pm\infty, 
\]
we can define the function
\[
w_M(t):=\max_{x\in\bbr} \we (t,x),
\]
which is Lipschitz continuous, and differentiable almost everywhere on $[t_1,t_2]$ thanks to the regularity $\we\in C^1([0,T]\times\bbr)$. 
Moreover, for each $t\in[t_1,t_2]$, there exists $x_t$ such that 
\[
w_M(t)=\we (t,x_t).
\]
Then $w'_M(t) =(\partial_t \we) (t,x_t)$ for a.e. $t\in (t_1,t_2)$, since
\begin{align*}
\begin{aligned} 
w'_M(t) &= \lim_{h\to 0+} \frac{\we(t+h, x_{t+h}) - \we (t,x_t)}{h} \\
&\ge  \lim_{h\to 0+} \frac{\we(t+h, x_{t}) - \we (t,x_t)}{h} = \partial_t \we (t,x_t),\\
w'_M(t) &= \lim_{h\to 0+} \frac{\we (t,x_t)-\we(t-h, x_{t-h})}{h}  \\
&\le  \lim_{h\to 0+} \frac{\we (t,x_t)-\we(t-h, x_{t})}{h} = \partial_t \we (t,x_t).
\end{aligned}
\end{align*}
Using this together with $\px^2 \we(t,x_t)\le 0$, $\px \we(t,x_t)=0$ and $\rhoe\mue'(\rhoe)\ge 0$, we have from \eqref{w-eq} that 
\[
w'_M (t) \le J_1(t) w_M(t) + J_2(t),\quad t\in (t_1,t_2),
\]
where (putting $\rho_M(t):=\rhoe(t,x_t)$) 
\begin{align*}
\begin{aligned} 
&J_1(t):= \frac{\rho_M^\gamma}{\mue(\rho_M)^2} \left( \gamma \mue(\rho_M) -2 \left( \rho_M\mue'(\rho_M)+\mue(\rho_M) \right) \right), \\
&J_2(t):= \frac{\rho_M^{2\gamma}}{\mue(\rho_M)^2} \left( \gamma \mue(\rho_M) - \left( \rho_M\mue'(\rho_M)+\mue(\rho_M) \right) \right) .
\end{aligned}
\end{align*}
Since $\gamma\le \alpha+1$, we have
\begin{align*}
\begin{aligned} 
J_1(t)&= \frac{\rho_M^\gamma}{\mue(\rho_M)^2} \left( (\gamma -2(\alpha+1))\rho_M^\alpha {\mathbf 1}_{\{\rho_M^\alpha>\eps\rho_M^{\alpha_*}\}} + \eps\left( \gamma-2(\alpha_*+1) \right)  \rho_M^{\alpha_*} {\mathbf 1}_{\{\rho_M^\alpha\le\eps\rho_M^{\alpha_*}\}}  \right)\\
&\le \frac{\rho_M^\gamma}{\mue(\rho_M)^2} \eps\left| \gamma- 2(\alpha_*+1)\right|  \rho_M^{\alpha_*} {\mathbf 1}_{\{\rho_M^\alpha\le\eps\rho_M^{\alpha_*}\}}.
\end{aligned}
\end{align*}
Moreover, using $\mue(\rho_M)\ge \eps \rho_M^{\alpha_*}$ and $\mue(\rho_M)\ge \rho_M^\alpha$ by the definition, we have
\[
J_1(t) \le \left| \gamma-2(\alpha_*+1) \right|  \rho_M^{\gamma-\alpha} {\mathbf 1}_{\{\rho_M^\alpha\le\eps\rho_M^{\alpha_*}\}} \le \left| \gamma-2(\alpha_*+1) \right| \eps^{\frac{\gamma-\alpha}{\alpha-\alpha_*}}.
\]
Likewise, we have
\begin{align*}
\begin{aligned} 
J_2(t)&= \frac{\rho_M^{2\gamma}}{\mue(\rho_M)^2} \left( (\gamma -(\alpha+1))\rho_M^\alpha {\mathbf 1}_{\{\rho_M^\alpha>\eps\rho_M^{\alpha_*}\}} + \eps\left( \gamma-(\alpha_*+1) \right)  \rho_M^{\alpha_*} {\mathbf 1}_{\{\rho_M^\alpha\le\eps\rho_M^{\alpha_*}\}}  \right)\\
&\le \frac{\rho_M^{2\gamma}}{\mue(\rho_M)^2} \eps\left| \gamma-(\alpha_*+1) \right|  \rho_M^{\alpha_*} {\mathbf 1}_{\{\rho_M^\alpha\le\eps\rho_M^{\alpha_*}\}}\\
&\le  \left| \gamma- (\alpha_*+1) \right| \eps^{\frac{2\gamma-\alpha}{\alpha-\alpha_*}}.
\end{aligned}
\end{align*}
The above estimates and \eqref{neg-le} imply that for any $t\in [t_1,t_2]$ and $\eps\in(0,1)$,
\begin{align}
\begin{aligned}\label{temp-wm} 
w_M(t) &\le w_M(t_1) \exp\left( \int_{t_1}^t J_1(s)ds \right) +  \int_{t_1}^t J_2(s) \exp\left( \int_{s}^t J_1(\tau) d\tau \right) ds\\
&\le   \exp\left( T\left| \gamma- 2(\alpha_*+1) \right| \right) \left( \eps^\theta + \eps^{\frac{2\gamma-\alpha}{\alpha-\alpha_*}} T \left| \gamma-  (\alpha_*+1) \right| \right) ,
\end{aligned}
\end{align}
If $\gamma>\alpha$, it follows from \eqref{temp-wm} that for all $\eps$ satisfying
\[
\eps\le \left(\frac{1}{1+ T \left| \gamma-(\alpha_*+1) \right|}\right)^{\frac{\alpha-\alpha_*}{\gamma-\alpha}},
\]
the following holds:
\[
w_M(t) \le 2\exp\left( T\left| \gamma- 2(\alpha_*+1) \right| \right) \eps^\theta, \quad \forall t\in [t_1,t_2].
\]
If $\gamma=\alpha$, since $\theta=\frac{2\gamma-\alpha}{\alpha-\alpha_*}$, it follows from \eqref{temp-wm} that 
\[
w_M(t) \le 2 \left(1+ T \left| \gamma-  (\alpha_*+1) \right| \right)  \exp\left( T\left| \gamma- 2(\alpha_*+1) \right| \right) \eps^\theta,\quad \forall\eps\le1, \quad \forall t\in [t_1,t_2].
\]
Therefore, the above estimates together with \eqref{neg-le} yield that 
\begin{align*}
\begin{aligned} 
\sup_{x\in\bbr}  w_\eps (t,x) \le 
C_\gamma \eps^\theta, \qquad \forall \eps\le \eps_\gamma, 
\quad \forall t\in [0,t_2],
\end{aligned}
\end{align*}
where $C_\gamma$  is the constants as in \eqref{special-c}.\\
If $t_2 <T$, then the definition of $t_2$ implies
\[
\sup_{x\in\bbr} \we (t,x) \le \eps^\theta, \quad \forall t\in (t_2, T].
\]
Hence we complete the proof.
\end{proof}

\begin{proposition}\label{lem:rho2}
Assume the same hypotheses as in Theorem \ref{thm:main}. Then, for any $T>0$, there exist positive constants $\underline\kappa(T)=\underline\kappa(T)(\gamma, \alpha, \underline \kappa_0)$ and $\delta_1=\delta_1(T,\gamma, \alpha, \underline \kappa_0)$ (independent of $\eps$) such that
\[
\rhoe (t,x) \ge \underline\kappa(T),\qquad \forall t\le T,\quad \forall x\in\bbr, \quad \forall \eps \le \delta_1 .
\]
\end{proposition}
\begin{proof}
Let
\[
q(\gamma):=\left\{ \begin{array}{ll}
\theta  \qquad \mbox{if } \gamma>\alpha, \\
 1 \qquad \mbox{if } \gamma=\alpha, \end{array} \right. \quad\mbox{where~} \theta=\frac{\gamma}{\alpha-\alpha_*} \mbox{ as in Lemma \ref{lem:up-w}}.
\]
We first choose  a constant $\delta_1>0$ such that
\beq\label{eps1}
\delta_1 := \left\{ \begin{array}{ll}
\min \left(\eps_\gamma, \left( \frac{\underline \kappa_0}{4} \right)^{\alpha-\alpha_* }, \left( \frac{2^\alpha -1}{ \alpha (2^\gamma+C_\gamma)T} \right)^{\frac{\gamma}{q(\gamma)(\gamma-\alpha)}} \right) \qquad \mbox{if } \gamma>\alpha, \\
 \min \left( \eps_\gamma, \left( \frac{\underline \kappa_0}{4} \right)^\alpha , \left( C_\gamma^{-1}(2^\alpha -1)  e^{-\alpha T} \right)^{\frac{\alpha-\alpha_*}{\alpha_*}}\right) \quad\quad \mbox{if } \gamma=\alpha, \end{array} \right.
\eeq
where $\underline \kappa_0$ is the constant as in \eqref{ass-ini}, and $\eps_\gamma, C_\gamma$ are the constants as in Lemma \ref{lem:up-w}.\\

Then, since 
\[
\delta_1 \le  \left\{ \begin{array}{ll}
\left( \frac{\underline \kappa_0}{4} \right)^{\alpha-\alpha_* }\qquad \mbox{if } \gamma>\alpha, \\
\left( \frac{\underline \kappa_0}{4} \right)^\alpha  \quad\qquad \mbox{if } \gamma=\alpha, \end{array} \right.
\]
we have $2\delta_1^{q(\gamma)/\gamma} < \underline \kappa_0$ for any $\gamma\ge\alpha$. \\
Therefore, it follows from the initial condition of \eqref{ass-ini} that
\[
\inf_{x\in\bbr} \rho_0 (x)\ge 2\delta_1^{q(\gamma)/\gamma}.
\]
For any fixed $\eps\le\delta_1$, since $\rhoe \in C([0,T]\times\bbr)$, if there exists a point $(t_0,x_0)\in (0,T]\times\bbr$ such that $\rhoe(t_0,x_0)<2\delta_1^{q(\gamma)/\gamma}$, 
then there exists $t_1\ge 0$ such that 
\beq\label{rhoet1}
\inf_{x\in\bbr} \rhoe (t,x)\ge 2\delta_1^{q(\gamma)/\gamma}\quad \forall t\in [0,t_1],
\eeq
\[
 \inf_{x\in\bbr} \rhoe (t,x)< 2 \delta_1^{q(\gamma)/\gamma}\quad \forall t\in (t_1,t_0].
\]
Then, 
\beq\label{inf-rhot}
 \inf_{x\in\bbr} \rhoe (t,x)\le 2 \delta_1^{q(\gamma)/\gamma} \quad \forall t\in [t_1,t_2],
\eeq
where 
\[
t_2:=\sup \left\{ t\in (t_1, T]~|~  \inf_{x\in\bbr} \rhoe (t,x)< 2\delta_1^{q(\gamma)/\gamma} \right\}.
\]
Thus, using $2 \delta_1^{q(\gamma)/\gamma} < \underline \kappa_0 \le \min(\rho_-,\rho_+)$  together with the fact that for each $t\le T$,
\[
\rhoe (t,x) \to \rho_\pm  \quad\mbox{as } ~ x\to\pm\infty, 
\]
we define the function
\[
\rho_m(t):=\min_{x\in\bbr} \rhoe (t,x),
\]
which is Lipschitz continuous, and differentiable almost everywhere on $[t_1,t_2]$ thanks to the regularity $\rhoe\in C^1([0,T]\times\bbr)$.
So, let $y_t$ be a minimizer for $\rho_m(t)=\rhoe (t,y_t)$. Since $\rho'_m(t) =(\partial_t \rhoe) (t,y_t)$ for a.e. $t\in (t_1,t_2)$, and $\px \rhoe(t,y_t)=0$, we have from the continuity equation of \eqref{NS-eps} that
\[
\rho'_m(t) = -\rho_m(t) \px \ue (y_t),\quad  t\in (t_1,t_2) .
\]
Then, using \eqref{def-w}, Lemma \ref{lem:up-w} with $\eps\le\delta_1\le\eps_\gamma$, and $\mue(\rho_m)\ge \rho_m^\alpha$, we have
\beq\label{rhoma}
\rho'_m(t) = -\rho_m(t) \frac{p(\rho_m)+\we(y_t)}{\mue(\rho_m)} \ge  -\rho_m^{1+\gamma-\alpha} - C_\gamma\delta_1^\theta \rho_m^{1-\alpha},\qquad   t\in (t_1,t_2) .
\eeq
{\bf Case of $\gamma>\alpha$)} 
Using \eqref{inf-rhot} together with $q(\gamma)=\theta$, we have
\[
\rho'_m \ge -  (2^\gamma+C_\gamma) \delta_1^\theta \rho_m^{1-\alpha},
\]
which yields
\[
(\rho_m^\alpha )' \ge - \alpha (2^\gamma+C_\gamma)  \delta_1^\theta,\quad   t\in (t_1,t_2) .
\]
Thus, using \eqref{rhoet1}, we have
\[
\rho_m^\alpha(t) \ge \rho_m^\alpha(t_1) - \alpha(2^\gamma+C_\gamma) \delta_1^\theta T\ge \left(2\delta_1^{q(\gamma)/\gamma} \right)^\alpha- \alpha(2^\gamma+C_\gamma)\delta_1^\theta T ,\quad  \forall t\in [t_1,t_2].
\]
Since $q(\gamma)=\theta$ when $\gamma>\alpha$, and
\[
\delta_1 \le \left( \frac{2^\alpha -1}{ \alpha (2^\gamma+C_\gamma)T} \right)^{\frac{\gamma}{q(\gamma)(\gamma-\alpha)}},
\]
we have
\[
\rho_m^\alpha (t) \ge  \left( \delta_1^{q(\gamma)/\gamma}  \right)^\alpha  ,\quad  \forall t\in [t_1,t_2].
\]
Therefore, this together with \eqref{rhoet1} and the definition of $t_2$ implies 
\[
\inf_{x\in\bbr} \rhoe (t,x)\ge \delta_1^{q(\gamma)/\gamma}\quad \forall t\in [0,T].
\]
{\bf Case of $\gamma=\alpha$)} 
First, it follows from \eqref{rhoma} with $\gamma=\alpha$ that
\[
\rho'_m\ge  -\rho_m - C_\gamma\delta_1^\theta \rho_m^{1-\alpha} ,\qquad   t\in (t_1,t_2) .
\]
Then, since
\[
(\rho_m^\alpha )' \ge -\alpha \rho_m^\alpha - \alpha C_\gamma\delta_1^\theta,\qquad   t\in (t_1,t_2) ,
\]
we have
\[
\rho_m^\alpha(t) \ge \rho_m^\alpha (t_1) e^{-\alpha(t-t_1)} -\alpha C_\gamma\delta_1^\theta \int_{t_1}^t e^{-\alpha(t-s)} ds,
\]
which together with  \eqref{rhoet1} yields
\[
\rho_m^\alpha(t) \ge \left(2\delta_1^{q(\gamma)/\gamma} \right)^\alpha e^{-\alpha T} - C_\gamma\delta_1^\theta,\qquad  \forall t\in [t_1,t_2].
\]
Since $q(\gamma)/\gamma= 1/\alpha$ and $\theta=\alpha/(\alpha-\alpha_*)$ when $\gamma=\alpha$, if needed, taking $\delta_1$ again such that 
\[
\delta_1 \le \left( C_\gamma^{-1}(2^\alpha -1)  e^{-\alpha T} \right)^{\frac{\alpha-\alpha_*}{\alpha_*}},
\]
we have
\[
\rho_m^\alpha(t) \ge  e^{-\alpha T} \delta_1 ,\quad  \forall t\in [t_1,t_2].
\]
Therefore, this together with \eqref{rhoet1} and the definition of $t_2$ implies 
\[
\inf_{x\in\bbr} \rhoe (t,x)\ge e^{-T} \delta_1^{1/\alpha} = e^{-T}  \delta_1^{q(\gamma)/\gamma}\quad \forall t\in [0,T].
\]
Hence we complete the proof.
\end{proof}

\subsection{Uniform bounds for the solutions $(\rhoe,\ue)$} \label{sec:all}

Thanks to Proposition \ref{lem:rho2}, we first have the uniform upper bound for the density as follows:
\begin{proposition}\label{lem:rho3}
Under the same hypotheses as in Theorem \ref{thm:main}, there exists a positive constant $\overline\kappa(T)$ (independent of $\eps$) such that
\[
\rhoe (t,x) \le \overline\kappa(T),\qquad \forall t\le T,\quad \forall x\in\bbr, \quad \forall \eps \le\delta_1,
\]
where $\delta_1$ is the constant as in Proposition \ref{lem:rho2}. 
\end{proposition}

For the proof of Proposition \ref{lem:rho3}, we refer to the proof of \cite[Proposition 4.5]{MV_sima}, in which the uniform estimates \eqref{kamue} and \eqref{2est-bd} are crucially used to get the uniform upper bound $ \overline\kappa(T)$ of the density:
One estimate is on the uniform lower bound of the viscosity $\mue$ as
\beq\label{kamue}
\mue(\rhoe) \ge \rhoe^\alpha \ge \underline\kappa(T)^\alpha ,\qquad \forall t\le T,\quad \forall x\in\bbr, \quad \forall \eps \le \delta_1.
\eeq
The others are the estimates \cite[Lemmas 3.1 and 3.2]{MV_sima} on the relative entropy related to the Bresch-Desjardins entropy (see \cite{BD_Paris02,BD_cmp03,BD_Paris04}) as follows:
\begin{align}
\begin{aligned}\label{2est-bd} 
&\sup_{0\le t\le T} \int_\bbr  \left(\rhoe \left| \ue-\bar u \right|^2 + p(\rhoe|\bar\rho) \right) dx + \int_0^T\int_\bbr \mue(\rhoe) |\px \ue|^2 dx dt \le K,\\
&\sup_{0\le t\le T} \int_\bbr  \left(\rhoe \left|( \ue-\bar u )+ \px (\varphi(\rhoe)) \right|^2 + p(\rhoe|\bar\rho) \right) dx \le K,\\
\end{aligned}
\end{align}
where $\varphi'(\rhoe):=\mue(\rhoe)/\rhoe^2$, and the above constant $K$ is independent of $\eps$ thanks to \eqref{up-mue}. Indeed, it follows from \cite[Lemmas 3.1 and 3.2]{MV_sima} that the constant $K$ depends only on $T, \gamma, (\bar\rho, \bar u), (\rho_0,u_0)$, and the constants appearing in \eqref{oldv-2}. \\

Propositions  \ref{lem:rho2} and \ref{lem:rho3} together with the above estimates \eqref{kamue}-\eqref{2est-bd} imply the following uniform estimates on the Sobolev norms of the solutions  $(\rhoe,\ue)$ :

\begin{proposition}\label{lem:all}
Under the same hypotheses as in Theorem \ref{thm:main}, there exists a constant $C$ (independent of $\eps$) such that
\[
\|\rho_\eps -\bar \rho \|_{L^\infty(0,T;H^k(\bbr))} + \| u_\eps-\bar u \|_{L^\infty(0,T;H^k(\bbr))} +  \| u_\eps-\bar u \|_{L^2(0,T;H^{k+1}(\bbr))} \le C.
\]
\end{proposition}

For the proof of proposition \ref{lem:all}, we first refer to the proof of \cite[Proposition 4.6 and 4.7]{MV_sima}, from which the constant in \eqref{eps-reg} does not depend on $\eps$ anymore.
Then, from the proof of Lemma \ref{lem:higher}, we deduce that the constant $C$ in Lemma \ref{lem:higher} is independent of $\eps$. Therefore, we have Proposition \ref{lem:all}

\subsection{Conclusion}
We have shown that for any $\eps\le\delta_1$, the system \eqref{NS-eps} has the unique smooth solution $(\rhoe,\ue)$  such that Propositions \ref{lem:rho2},  \ref{lem:rho3} and  \ref{lem:all} hold.\\
We now take $\delta_T$ as
\[
\delta_T  = \min\left( \underline\kappa(T)^{\alpha-\alpha_*} ,  \delta_1 \right),
\]
where the constants $\underline\kappa(T)$ and $\delta_1$ are as in Proposition \ref{lem:rho2}.\\
Then, since Proposition \ref{lem:rho2} implies that for all $ \eps <\delta_T$,
\[
\eps\rhoe^{\alpha_*} < \delta_T \rhoe^{\alpha_*} \le \underline\kappa(T)^{\alpha-\alpha_*} \rhoe^{\alpha_*} \le \rhoe^\alpha ,\qquad  \forall t\le T,\quad \forall x\in\bbr,
\]
it follows from the definition \eqref{mu-def} that
\beq\label{ind-mu}
\mue(\rhoe)=\mu(\rhoe), \qquad  \forall \eps <\delta_T,\quad  \forall t\le T,\quad \forall x\in\bbr.
\eeq
Recall that the approximate system \eqref{NS-eps} represents the system \eqref{NS} with $\mue$ instead of $\mu$.\\
Therefore, for any $T>0$, and any $\eps$ with $\eps<\delta_T$, $(\rhoe,\ue)$ is the unique smooth solution of \eqref{NS} with the initial datum $(\rho_0, u_0)$ such that Propositions \ref{lem:rho2},  \ref{lem:rho3} and  \ref{lem:all} hold.\\
Hence we complete the proof.

\begin{appendix}
\setcounter{equation}{0}

\section{Proof of Lemma \ref{lem:higher}} \label{app-higher}
Let $(\rho_\eps, u_\eps)$ be the global strong solution to \eqref{NS-eps} such that \eqref{eps-reg} and \eqref{eps-bdd} hold.\\
Once the desired estimates for $k=2$ are obtained, the remaining part proceeds by induction in $k$, which follows the same proof of \cite[Lemma 4.3]{CDNP}. 
Therefore, we here present the proof only when $k=2$, based on the proof of \cite[Lemma 4.2]{CDNP}.\\

First of all, since $\partial_x u_\eps \in L^2(0,T;L^\infty(\bbr))$ by \eqref{eps-reg}, using \eqref{eps-reg} and \eqref{eps-bdd}, we have
\begin{align}
\begin{aligned} \label{reg-w}
&w_\eps \in L^2(0,T;L^\infty(\bbr)), \\
&\partial_x w_\eps = -p'(\rho_\eps) \partial_x \rho_\eps +\mu_\eps'(\rho_\eps) \partial_x \rho_\eps \partial_x u_\eps +\mu_\eps(\rho_\eps) \partial_x^2 u_\eps \in L^2(0,T;L^2(\bbr)).
\end{aligned}
\end{align}
{\bf Step 1)} 
Differentiating the equation \eqref{w-eq} in space, multiplying the resulting equation by $\px\we$ and integrating by parts, we have
\begin{align*}
\begin{aligned} 
\frac{d}{dt} \int_\bbr \frac{|\px \we|^2}{2} dx &= -\int_\bbr  \frac {\mue(\rhoe)}{\rhoe} |\px^2 \we|^2 dx +\int_\bbr \left(\ue+\frac{\mue(\rhoe)}{\rhoe^2} \px\rhoe\right) \px\we \px^2\we dx\\
&~ + \int_\bbr f_1(\rhoe) |\px\we|^2 dx +  \int_\bbr f_1'(\rhoe) \px \rhoe \we \px\we dx - 2\int_\bbr f_2(\rhoe) \we |\px\we|^2 dx \\
&~ - \int_\bbr f_2'(\rhoe) \px\rhoe \we^2 \px\we dx +\int_\bbr f_3'(\rhoe) \px\rhoe \px \we dx\\
&=: -\int_\bbr  \frac {\mue(\rhoe)}{\rhoe} |\px^2 \we|^2 dx + \sum_{j=1}^6 I_j.
\end{aligned}
\end{align*}
where 
\begin{align*}
\begin{aligned} 
&f_1(\rho):= \rho\frac{p'(\rho)}{\mue(\rho)} -2p(\rho) \frac{\rho\mue'(\rho)+\mue(\rho)}{\mue(\rho)^2},\\
& f_2(\rho):=\frac{\rho\mue'(\rho)+\mue(\rho)}{\mue(\rho)^2},\\
& f_3(\rho):=\left( \rho\frac{p'(\rho)}{\mue(\rho)} -p(\rho)\frac{\rho\mue'(\rho)+\mue(\rho)}{\mue(\rho)^2}  \right) p(\rho). 
\end{aligned}
\end{align*}
Since, thanks to \eqref{eps-bdd}, $L^\infty([0,T]\times\bbr)$-norms of $\rhoe$ to some power are all bounded, there exists a positive constant $C_1=C_1(\underline \kappa_\eps(T), \overline \kappa_\eps(T))$ such that
\[
-\int_\bbr  \frac {\mue(\rhoe)}{\rhoe} |\px^2 \we|^2 dx \le - C_1 \int_\bbr  |\px^2 \we|^2 dx,
\]
and 
\[
\left\|\frac{\mue(\rhoe)}{\rhoe^2} \right\|_{L^\infty([0,T]\times\bbr)} + \sum_{j=1}^3\left(\|f_j(\rhoe) \|_{L^\infty([0,T]\times\bbr)}+ \|f_j'(\rhoe) \|_{L^\infty([0,T]\times\bbr)} \right) \le C_1.
\]
Thus, the above terms $I_j$ can be controlled as follows:
\begin{align*}
\begin{aligned} 
|I_1|&\le \|\ue\|_{L^\infty(\bbr)}\|\px\we \|_{L^2(\bbr)} \| \px^2\we\|_{L^2(\bbr)} + C_1 \|\px \rhoe\|_{L^\infty(\bbr)} \|\px\we \|_{L^2(\bbr)} \| \px^2\we\|_{L^2(\bbr)} \\
&\le \frac{C_1}{2}  \| \px^2\we\|_{L^2(\bbr)}^2 +  C  \left(\|\ue\|_{L^\infty(\bbr)}^2 + \|\px\rhoe \|_{L^2(\bbr)}^2 +  \|\px^2\rhoe \|_{L^2(\bbr)}^2  \right) \|\px\we \|_{L^2(\bbr)}^2,
\end{aligned}
\end{align*}
\begin{align*}
\begin{aligned} 
|I_2|&\le C_1 \|\px\we \|_{L^2(\bbr)}^2,\\
|I_3|&\le C_1\|\px\rhoe \|_{L^2(\bbr)}  \|\we\|_{L^\infty(\bbr)} \|\px\we \|_{L^2(\bbr)} \le C_1\|\px\rhoe \|_{L^2(\bbr)} \left(\|\we\|_{L^\infty(\bbr)}^2 + \|\px\we \|_{L^2(\bbr)}^2  \right),\\
|I_4|&\le 2 C_1 \|\we\|_{L^\infty(\bbr)}  \|\px\we \|_{L^2(\bbr)}^2,\\
|I_5|&\le C_1 \|\px\rhoe \|_{L^2(\bbr)} \|\we\|_{L^\infty(\bbr)}^2 \|\px\we \|_{L^2(\bbr)} \\
&\le C_1\|\px\rhoe \|_{L^2(\bbr)} \left(\|\we\|_{L^\infty(\bbr)}^2 +\|\we\|_{L^\infty(\bbr)}^2 \|\px\we \|_{L^2(\bbr)}^2  \right) ,\\
|I_6|&\le C_1\|\px\rhoe \|_{L^2(\bbr)}^2 + C_1 \|\px\we \|_{L^2(\bbr)}^2.
\end{aligned}
\end{align*}
Moreover, since it follows from \eqref{eps-reg} and $\bar\rho\in L^\infty(\bbr)$ that 
\beq\label{int-rho}
\px\rhoe \in L^\infty(0,T; L^2(\bbr)) \quad\mbox{and}\quad \ue  \in L^\infty(0,T; L^\infty(\bbr)),
\eeq
we have
\begin{align}
\begin{aligned}\label{dw-est} 
\frac{d}{dt}\| \px\we\|_{L^2(\bbr)}^2 +  C_1  \| \px^2\we\|_{L^2(\bbr)}^2 \le C\left( 1+\| \px^2 \rhoe\|_{L^2(\bbr)}^2 + \| \we\|_{L^\infty(\bbr)}^2   \right) \| \px\we\|_{L^2(\bbr)}^2 +F,
\end{aligned}
\end{align}
where 
\[
F=C\left( 1+\| \we\|_{L^\infty(\bbr)}^2   \right).
\]
Note from \eqref{reg-w} that $F\in L^1((0,T))$. \\
{\bf Step 2)} 
We next estimate $\| \px^2 \rhoe\|_{L^2(\bbr)}$, to control $\| \px^2 \rhoe\|_{L^2(\bbr)}^2$ in \eqref{dw-est}.\\
Differentiating the continuity equation of \eqref{NS-eps} twice in space, and multiplying the resulting equation by $\px^2\rhoe$, we have
\begin{align*}
\begin{aligned} 
\frac{d}{dt} \int_\bbr \frac{|\px^2 \rhoe|^2}{2} dx &= - \int_\bbr \px^2(\ue \px\rhoe) \px^2\rhoe dx - \int_\bbr \px^2(\rhoe\px\ue)  \px^2\rhoe dx\\
&= -\int_\bbr \ue \px \left(\frac{|\px^2\rhoe|^2}{2}\right) dx -\int_\bbr  \underbrace{\left( \px^2(\ue \px\rhoe)- \ue  \px^2 \px \rhoe \right)}_{=:J_1}  \px^2\rhoe dx \\
&\quad - \int_\bbr \rhoe\px^3 \ue  \px^2\rhoe dx -\int_\bbr  \underbrace{\left( \px^2(\rhoe\px\ue)  -\rhoe\px^3 \ue \right)}_{=:J_2}  \px^2\rhoe  dx.
\end{aligned}
\end{align*}
Using the commutator estimates \cite[Lemma 3.4]{MB} and the Sobolev embedding, we have
\begin{align*}
\begin{aligned} 
\| J_1 \|_{L^2(\bbr)} &\le C \| \px^2 \ue \|_{L^2(\bbr)} \| \px \rhoe \|_{L^\infty(\bbr)}+ C\| \px \ue \|_{L^\infty (\bbr)} \| \px^2 \rhoe \|_{L^2(\bbr)}\\
&\le C \| \px^2 \ue \|_{L^2(\bbr)} \| \px \rhoe \|_{H^1(\bbr)}+ C\| \px \ue \|_{H^1 (\bbr)} \| \px^2 \rhoe \|_{L^2(\bbr)},\\
\| J_2 \|_{L^2(\bbr)} &\le C \| \px^2 \rhoe \|_{L^2(\bbr)} \| \px \ue \|_{L^\infty(\bbr)}+ C\| \px \rhoe \|_{L^\infty (\bbr)} \| \px^2 \ue \|_{L^2(\bbr)}\\
&\le C \| \px^2 \rhoe \|_{L^2(\bbr)} \| \px \ue \|_{H^1(\bbr)}+ C\| \px \rhoe \|_{H^1 (\bbr)} \| \px^2 \ue \|_{L^2(\bbr)}.
\end{aligned}
\end{align*}
Therefore, we have
\begin{align*}
\begin{aligned} 
\frac{d}{dt} \int_\bbr \frac{|\px^2 \rhoe|^2}{2} dx &\le \frac{1}{2} \| \px \ue \|_{L^\infty(\bbr)} \| \px^2 \rhoe \|_{L^2(\bbr)}^2 +  \| \rhoe \|_{L^\infty(\bbr)} \| \px^3 \ue \|_{L^2(\bbr)}
\| \px^2 \rhoe \|_{L^2(\bbr)}\\
&\quad +C\left( \| \px^2 \ue \|_{L^2(\bbr)} \| \px \rhoe \|_{L^2(\bbr)} +\| \px \ue \|_{H^1(\bbr)} \| \px^2 \rhoe \|_{L^2(\bbr)} \right) \| \px^2 \rhoe \|_{L^2(\bbr)}.
\end{aligned}
\end{align*}
Moreover, using \eqref{eps-bdd}, \eqref{int-rho} and the Sobolev embedding, we have
\begin{align}
\begin{aligned} \label{d2rho-est}
\frac{d}{dt} \|\partial_x^2 \rhoe\|_{L^2(\bbr)}^2 &\le C\left( \| \partial_x \ue \|_{H^1(\bbr)} + \| \partial_x^2 \ue \|_{L^2(\bbr)}^2   \right) \| \partial_x^2 \rhoe \|_{L^2(\bbr)}^2 \\
&\quad +  C \|\partial_x^3 \ue \|_{L^2(\bbr)} \| \partial_x^2 \rhoe \|_{L^2(\bbr)} +C .
\end{aligned}
\end{align}

To estimate  $\|\partial_x^3 \ue \|_{L^2(\bbr)}$ in \eqref{d2rho-est}, we use the definition \eqref{def-w} of $\we$ as follows:
\beq\label{fordu}
\px \ue = g(\rhoe) \we + h(\rhoe),\quad \mbox{where }~ g(\rhoe):=\frac{1}{\mue(\rhoe)},~h(\rhoe):=\frac{p(\rhoe)}{\mue(\rhoe)}.
\eeq
Since 
\begin{align*}
\begin{aligned} 
\px^3 \ue &=g''(\rhoe) |\px\rhoe|^2 \we + g'(\rhoe) \px^2\rhoe \we+2g'(\rhoe) \px\rhoe\px \we  + g(\rhoe) \px^2 \we  \\
&\quad +  h''(\rhoe)  |\px\rhoe|^2 + h'(\rhoe) \px^2\rhoe ,
\end{aligned}
\end{align*}
we use \eqref{eps-bdd} to have
\begin{align}
\begin{aligned} \label{d3u}
 \|\px^3 \ue\|_{L^2(\bbr)} &\le C\Big( \big(\|\we \|_{L^\infty(\bbr)} +1\big)  \|\px\rhoe \|_{L^\infty(\bbr)} \|\px\rhoe\|_{L^2(\bbr)} +  \|\we \|_{L^\infty(\bbr)}  \|\px^2\rhoe\|_{L^2(\bbr)} \\
&\qquad + \|\px\rhoe \|_{L^\infty(\bbr)} \|\px\we\|_{L^2(\bbr)} + \|\px^2\we\|_{L^2(\bbr)} + \|\px^2\rhoe\|_{L^2(\bbr)}  \Big).
\end{aligned}
\end{align}
Combining this with \eqref{d2rho-est}, and using  \eqref{int-rho} and the Sobolev embedding, we have
\beq\label{d2rho}
\frac{d}{dt} \|\partial_x^2 \rhoe\|_{L^2(\bbr)}^2 \le \frac{C_1}{2} \|\px^2\we\|_{L^2(\bbr)}^2 +G_1 \|\partial_x^2 \rhoe \|_{L^2(\bbr)}^2 + G_2,
\eeq
where 
\begin{align*}
\begin{aligned} 
G_1 &:= C \left( \| \partial_x \ue \|_{H^1(\bbr)} + \| \partial_x^2 \ue \|_{L^2(\bbr)}^2 +\|\we \|_{L^\infty(\bbr)} + \|\px \we\|_{L^2(\bbr)}  +1  \right),\\
G_2 &:=  C \left(\|\we \|_{L^\infty(\bbr)}^2 + \|\px \we\|_{L^2(\bbr)}^2 +1  \right).
\end{aligned}
\end{align*}
Note that $G_1, G_2 \in L^1((0,T))$ by \eqref{eps-reg} and \eqref{reg-w}.\\
{\bf Step 3)}
Adding \eqref{dw-est} to \eqref{d2rho}, we have
\begin{align*}
\begin{aligned}
&\frac{d}{dt}\left(\|\px\we\|_{L^2(\bbr)}^2+\|\partial_x^2 \rhoe\|_{L^2(\bbr)}^2 \right) +  \frac{C_1}{2}  \| \px^2\we\|_{L^2(\bbr)}^2 \\
&\qquad\quad \le H \left(\|\px\we\|_{L^2(\bbr)}^2+\|\partial_x^2 \rhoe\|_{L^2(\bbr)}^2 \right) + F+G_2,
\end{aligned}
\end{align*}
where
\[
H:=C\left( 1+\| \px\we\|_{L^2(\bbr)}^2+ \| \we\|_{L^\infty(\bbr)}^2 + \| \partial_x \ue \|_{H^1(\bbr)} + \| \partial_x^2 \ue \|_{L^2(\bbr)}^2  \right).
\]
Since $H, F, G_2\in L^1((0,T))$, and it follows from \eqref{def-w} and \eqref{temp-ini} that
\[
\|\px\we(0)\|_{L^2(\bbr)} \le C(\underline\kappa_0, \overline\kappa_0)\left(  \|\px\rho_0\|_{L^2(\bbr)} + \|\px\rho_0\|_{L^2(\bbr)} \|\px u_0\|_{L^2(\bbr)}+\|\px^2 u_0\|_{L^2(\bbr)}   \right) ,
\]
Gr$\ddot{\mbox{o}}$nwall lemma implies that
\beq\label{f-rhow}
\|\partial_x^2 \rho_\eps \|_{L^\infty(0,T;L^2(\bbr))} + \|\px w_\eps \|_{L^\infty(0,T;L^2(\bbr))}+\|\px^2 w_\eps \|_{L^2(0,T;L^2(\bbr))} \le C,
\eeq
where the constant $C>0$ depends on $T$ and the bounds of \eqref{eps-reg}, \eqref{eps-bdd} and \eqref{temp-ini}.\\
This now together with \eqref{reg-w}, \eqref{int-rho} and \eqref{d3u} imply the bound for $\px^3 \ue$:
\[
\|\px^3 \ue \|_{L^2(0,T;L^2(\bbr))} \le C.
\]
Moreover, differentiating the both sides of \eqref{fordu} in $x$, and using \eqref{eps-bdd}, we have
\[
\|\px^2 \ue \|_{L^2(\bbr)} \le C\Big( \|\px\rhoe\|_{L^2(\bbr)} \|\we \|_{L^\infty(\bbr)}+  \|\px\we \|_{L^2(\bbr)} + \|\px\rhoe\|_{L^2(\bbr)} \Big).
\]
Therefore, we use \eqref{eps-reg}, \eqref{eps-bdd} and \eqref{f-rhow} to have
\[
\|\px^2 \ue \|_{L^\infty(0,T;L^2(\bbr))} \le C.
\]
Indeed, since  it follows from \eqref{eps-reg} and \eqref{eps-bdd} that
\[
\we = -p(\rhoe) +\mue(\rhoe)\px\ue \in L^\infty((0,T)\times\bbr)+ L^\infty(0,T;L^2(\bbr)),
\]
we use \eqref{f-rhow} to have
\begin{align*}
\begin{aligned}
|\we(x)|&\le \frac{1}{2}\int_{x-1}^{x+1} (|p(\rhoe)| +|\mue(\rhoe)\px\ue|) dy +  \frac{1}{2}\int_{x-1}^{x+1}\int_y^x |\partial_z \we| dz dy\\
&\le \|p(\rhoe) \|_{L^\infty((0,T)\times\bbr)} + \frac{1}{\sqrt{2}} \|\mue(\rhoe)\px\ue \|_{L^\infty(0,T;L^2(\bbr))} + \sqrt{2}  \|\px w_\eps \|_{L^\infty(0,T;L^2(\bbr))},
\end{aligned}
\end{align*}
which gives $\|\we \|_{L^\infty((0,T)\times\bbr)}\le C$.\\
Hence we complete the proof.


\end{appendix}

\bibliography{Kang-Vasseur2015}

\begin{thebibliography}{10}

\bibitem{BD_Paris02}
D.~Bresch and B.~Desjardins.
\newblock {Sur un mod\`ele de Saint-Venant visqueux et sa limite
  quasi-g\'eostrophique}.
\newblock {\em C. R. Math. Acad. Sci. Paris}, 335:1079--1084, 2002.

\bibitem{BD_cmp03}
D.~Bresch and B.~Desjardins.
\newblock Existence of global weak solutions for 2d viscous shallow water
  equations and convergence to the quasi-geostrophic model.
\newblock {\em Comm. Math. Phys.}, 238:211--223, 2003.

\bibitem{BD_Paris04}
D.~Bresch and B.~Desjardins.
\newblock {Some diffusive capillary models of Korteweg type}.
\newblock {\em C. R. Math. Acad. Sci. Paris, Section M\'ecanique},
  332:881--886, 2004.

\bibitem{CC}
S.~Chapman and T.G. Cowling.
\newblock The mathematical theory of non-uniform gases.
\newblock {\em Cambridge University Press, London}, 3rd ed., 1970.

\bibitem{CDNP}
P.~Constantin, T.~D. Drivas, H.~Q. Nguyen, and F.~Pasqualotto.
\newblock Compressible fluids and active potentials,.
\newblock {\em Annales de l'Institut Henri Poincar\'e. Analyse Non Lin\'eaire,
  To appear}, 2019.

\bibitem{GP}
J.-F. Gerbeau and B.~Perthame.
\newblock Derivation of viscous saint-venant system for laminar shallow water;
  numerical validation,.
\newblock {\em Discrete Contin. Dyn. Syst. Ser. B}, 1(1):89--102, 2018.

\bibitem{Haspot}
B.~Haspot.
\newblock Existence of global strong solution for the compressible
  {Navier-Stokes equations with degenerate viscosity coefficients in 1D},.
\newblock {\em Mathematische Nachrichten}, 291:2188--2203, 2018.

\bibitem{Hoff87}
D.~Hoff.
\newblock {Global existence for 1D, compressible, isentropic Navier-Stokes
  equations with large initial data}.
\newblock {\em Trans. Amer. Math. Soc}, 303:169--181, 1987.

\bibitem{Hoff98}
D.~Hoff.
\newblock {Global solutions of the equations of one-dimensional, compressible
  flow with large data and forces, and with differing end states}.
\newblock {\em Z. Angew. Math. Phys.}, 49:774--785, 1987.

\bibitem{HoSm}
D.~Hoff and J.~Smoller.
\newblock {Non-formation of vacuum states for compressible Navier-Stokes
  equations}.
\newblock {\em Comm. Math. Phys.}, 216:255--276, 2001.

\bibitem{Kang-V-NS17}
M.-J. Kang and A.~Vasseur.
\newblock {Contraction property for large perturbations of shocks of the
  barotropic {Navier-Stokes} system}.
\newblock {\em J. Eur. Math. Soc. (JEMS), To appear.
  https://arxiv.org/pdf/1712.07348.pdf}.

\bibitem{KV-unique19}
M.-J. Kang and A.~Vasseur.
\newblock {Uniqueness and stability of entropy shocks to the isentropic Euler
  system in a class of inviscid limits from a large family of Navier-Stokes
  systems}.
\newblock {\em https://arxiv.org/pdf/1902.01792.pdf}.

\bibitem{KS}
A.~V. Kazhikhov and V.~V. Shelukhin.
\newblock Unique global solution with respect to time of initial- boundary
  value problems for one-dimensional equations of a viscous gas,.
\newblock {\em Prikl. Mat. Meh.}, 41:282--291, 1977.

\bibitem{MB}
A.~Majda and A.~Bertozzi.
\newblock {\em Vorticity and Incompressible Flow}.
\newblock Cambridge Univ. Press, 2002.

\bibitem{MV_sima}
A.~Mellet and A.~Vasseur.
\newblock Existence and uniqueness of global strong solutions for
  one-dimensional compressible {Navier-Stokes equations}.
\newblock {\em SIAM J. Math. Anal.}, 39(4):1344--1365, 2007/08.

\bibitem{Serre86}
D.~Serre.
\newblock {Solutions faibles globales des \'equations de Navier-Stokes pour un
  fluide compressible}.
\newblock {\em C. R. Acad. Sci. Paris S\'er. I Math.}, 303:639--642, 1986.

\bibitem{Shel82}
V.~V. Shelukhin.
\newblock Motion with a contact discontinuity in a viscous heat conducting
  gas,.
\newblock {\em Dinamika Sploshn. Sredy}, 57:131--152, 1982.

\bibitem{Shel83}
V.~V. Shelukhin.
\newblock Evolution of a contact discontinuity in the barotropic flow of a
  viscous gas,.
\newblock {\em Prikl. Mat. Mekh.}, 47:870--872, 1983.

\bibitem{Shel86}
V.~V. Shelukhin.
\newblock Boundary value problems for equations of a barotropic viscous gas
  with nonnegative initial density.
\newblock {\em Dinamika Sploshn. Sredy}, 74:108--125, 1986.

\bibitem{Shel84}
V.V. Shelukhin.
\newblock On the structure of generalized solutions of the one-dimensional
  equations of a polytropic viscous gas,.
\newblock {\em J. Appl. Math. Mech., 48(1984), 665--672; translated from Prikl.
  Mat. Mekh. 48(1984), no. 6, 912--920}.

\bibitem{Solo}
V.~A. Solonnikov.
\newblock {The solvability of the initial-boundary value problem for the
  equations of motion of a viscous compressible fluid}.
\newblock {\em Zap. Nau$\check{c}$n. Sem. Leningrad. Otdel. Mat. Inst.
  Steklov.}, 59:128--142, 1976.

\bibitem{Vai}
V.~A. Vaigant.
\newblock Nonhomogeneous boundary value problems for equations of a viscous
  heat- conducting gas.
\newblock {\em Dinamika Sploshn. Sredy}, 97:3--21, 1990.

\end{thebibliography}
\end{document}